\documentclass{article}

\usepackage{latexsym,amsmath,amsfonts,amsthm,amssymb}

\numberwithin{equation}{section}

\newtheorem{Th}{Theorem}[section]
\newtheorem{Cor}[Th]{Corollary}

\newtheorem{Lemma}[Th]{Lemma}

\theoremstyle{definition}
\newtheorem{Example}[Th]{Example}
\newtheorem{df}[Th]{Definition}

\theoremstyle{remark}

\newcommand{\R}{\mathbb{R}}

\newcommand{\N}{\mathbb{N}}

\begin{document}

\title{The diamond-alpha Riemann integral
and mean value theorems on time scales}

\author{Agnieszka B. Malinowska${}^{\dag}$\\
        \texttt{abmalina@pb.bialystok.pl}
        \and
        Delfim F. M. Torres${}^{\ddag}$\\
        \texttt{delfim@ua.pt}}

\date{${}^{\dag}$Faculty of Computer Science\\
      Bia{\l}ystok Technical University\\
      15-351 Bia\l ystok, Poland\\[0.3cm]
      ${}^{\ddag}$Department of Mathematics\\
      University of Aveiro\\
      3810-193 Aveiro, Portugal}

\maketitle


\begin{abstract}
We study diamond-alpha integrals on time scales.
A diamond-alpha version of Fermat's theorem for stationary points
is also proved, as well as Rolle's, Lagrange's, and
Cauchy's mean value theorems on time scales.
\end{abstract}

\smallskip

\noindent \textbf{Mathematics Subject Classification 2000:}
26A42; 39A12.

\smallskip


\smallskip

\noindent \textbf{Keywords:} time scales, diamond-alpha integral,
Fermat's theorem for stationary points, mean value theorems.


\section{Introduction}

The calculus on time scales has been initiated
by Aulbach and Hilger in order to create a theory that
can unify and extend discrete and continuous analysis \cite{Aulbach-Hilger}. Two versions of the calculus on time scales, the delta and nabla calculus, are now standard in the theory of time scales \cite{livro,livro2}.
In 2006, a combined diamond-alpha dynamic derivative was introduced by Sheng, Fadag, Henderson, and Davis \cite{sfhd},
as a linear combination of the delta and nabla dynamic
derivatives on time scales.
The diamond-alpha derivative reduces
to the standard delta-derivative for $\alpha =1$ and to the standard nabla derivative for $\alpha =0$. On the other hand, it
represents a weighted dynamic derivative formula on any uniformly
discrete time scale when $\alpha =\frac{1}{2}$.
We refer the reader to
\cite{Controlo2008,RS,Sheng,sfhd,AFT} for a complete account
of the recent diamond-alpha calculus on time scales.
In Section~\ref{sec:prl} we briefly review the
necessary definitions and calculus on time scales;
our results are given in Section~\ref{sec:MR}.

The diamond-alpha integral on time scales is defined in
\cite{RS,Sheng,sfhd,AFT} by means of a
linear combination of the delta and nabla integrals.
In the present paper we use a Darboux approach to define the
Riemann diamond-alpha integral on time scales
and to prove the corresponding main theorems of the diamond-alpha integral calculus (Section~\ref{subsec:RI}). In addition,
we briefly investigate diamond-alpha improper integrals (Section~\ref{subsec:II}), and prove some new versions
of mean value theorems on time scales
via diamond-$\alpha$ derivatives and integrals
(Section~\ref{subsec:MVT}). A new notion of local extremum on time scales is also proposed, which leads to a
diamond-alpha Fermat's theorem for stationary points
(Theorem~\ref{thm:ft}) more similar in aspect to the classical condition than those of delta or nabla derivatives.


\section{Preliminaries on time scales}
\label{sec:prl}

In this section we introduce basic definitions and results from the theory of delta, nabla, and diamond-alpha time scales \cite{livro,RS,sfhd}.

A nonempty closed subset of $\mathbb{R}$ is called a \emph{time
scale} and is denoted by $\mathbb{T}$.
The \emph{forward jump operator}
$\sigma:\mathbb{T}\rightarrow\mathbb{T}$ is defined by
$$\sigma(t)=\inf{\{s\in\mathbb{T}:s>t}\},
\mbox{ for all $t\in\mathbb{T}$},$$
while the \emph{backward jump operator}
$\rho:\mathbb{T}\rightarrow\mathbb{T}$ is defined by
$$\rho(t)=\sup{\{s\in\mathbb{T}:s<t}\},\mbox{ for all
$t\in\mathbb{T}$},$$ with $\inf\emptyset=\sup\mathbb{T}$
(\textrm{i.e.}, $\sigma(M)=M$ if $\mathbb{T}$ has a maximum $M$),
and $\sup\emptyset=\inf\mathbb{T}$ (\textrm{i.e.}, $\rho(m)=m$ if
$\mathbb{T}$ has a minimum $m$).

A point $t\in\mathbb{T}$ is called \emph{right-dense},
\emph{right-scattered}, \emph{left-dense} and
\emph{left-scattered} if $\sigma(t)=t$, $\sigma(t)>t$, $\rho(t)=t$
and $\rho(t)<t$, respectively.

Throughout the paper we let $\mathbb{T}=[a,b]\cap\mathbb{T}_{0}$
with $a<b$ and $\mathbb{T}_0$ a time scale
containing $a$ and $b$.

The \emph{delta graininess function}
$\mu:\mathbb{T}\rightarrow[0,\infty)$ is defined by
$$\mu(t)=\sigma(t)-t,\mbox{ for all $t\in\mathbb{T}$} \, ;$$
the \emph{nabla graininess function}
is defined by $\eta(t):=t-\rho(t)$.

We introduce the sets $\mathbb{T}^{k}$, $\mathbb{T}_{k}$, and
$\mathbb{T}^{k}_{k}$, which are derived from the time scale
$\mathbb{T}$, as follows. If $\mathbb{T}$ has a left-scattered
maximum $t_{1}$, then $\mathbb{T}^{k}= \mathbb{T}-\{t_{1} \}$,
otherwise $\mathbb{T}^{k}= \mathbb{T}$. If $\mathbb{T}$ has a
right-scattered minimum $t_{2}$, then $\mathbb{T}_{k}=
\mathbb{T}-\{t_{2} \}$, otherwise $\mathbb{T}_{k}= \mathbb{T}$.
Finally, we define $\mathbb{T}_{k}^{k}= \mathbb{T}^{k} \cap
\mathbb{T}_{k}$.

We say that a function $f:\mathbb{T}\rightarrow\mathbb{R}$ is
\emph{delta differentiable} at $t\in\mathbb{T}^k$ if there exists
a number $f^{\Delta}(t)$ such that for all $\varepsilon>0$ there
is a neighborhood $U$ of $t$ (\textrm{i.e.},
$U=(t-\delta,t+\delta)\cap\mathbb{T}$ for some $\delta>0$) such
that
$$|f(\sigma(t))-f(s)-f^{\Delta}(t)(\sigma(t)-s)|
\leq\varepsilon|\sigma(t)-s|,\mbox{ for all $s\in U$}.$$ We call
$f^{\Delta}(t)$ the \emph{delta derivative} of $f$ at $t$ and say
that $f$ is \emph{delta differentiable} on $\mathbb{T}^k$ provided
$f^{\Delta}(t)$ exists for all $t\in\mathbb{T}^k$.

We define $f^{\nabla}(t)$
to be the number value, if one exists, such that for all $\epsilon
>0$, there is a neighborhood $V$ of $t$ such that for all $s \in
V$,
$$
\left|f(\rho(t))- f(s)-f^{\nabla}(t)\left(\rho(t)-s\right)\right|
\leq \epsilon |\rho(t)-s|.
$$
We say that $f$ is nabla differentiable on $\mathbb{T}_{k}$,
provided that $f^{\nabla}(t)$ exists for all $t \in
\mathbb{T}_{k}$.

For delta differentiable functions $f$ and $g$, the next formula
holds:
\begin{equation*}
\begin{aligned}
(fg)^\Delta(t)&=f^\Delta(t)g^\sigma(t)+f(t)g^\Delta(t)\\
&=f^\Delta(t)g(t)+f^\sigma(t)g^\Delta(t),
\end{aligned}
\end{equation*}
where we abbreviate here and throughout the text $f\circ\sigma$ by $f^\sigma$. Similarly property holds for nabla derivatives
(and we then use the notation $f^\rho = f\circ\rho$).

A function $f:\mathbb{T}\rightarrow\mathbb{R}$ is said to be a
\emph{regulated} function if its left-sided limits exist at left-dense points, and its right-sided limits exist at right-dense points.

A function $f:\mathbb{T}\rightarrow\mathbb{R}$ is called
\emph{rd-continuous} if it is continuous at right-dense points and if its left-sided limit exists at left-dense points. We denote the set of all rd-continuous functions by C$_{\textrm{rd}}$ and the
set of all delta differentiable functions with rd-continuous derivative
by C$_{\textrm{rd}}^1$.

Analogously, a function $f: \mathbb{T} \rightarrow \mathbb{R}$ is called ld-continuous, provided it is continuous at all left-dense
points in $\mathbb{T}$ and its right-sided limits exist finite at all right-dense points in $\mathbb{T}$.

It is known that rd-continuous functions possess a
\emph{delta antiderivative}, \textrm{i.e.}, there exists a function $F$
with $F^\Delta=f$, and in this case the \emph{delta integral} is
defined by $\int_{c}^{d}f(t)\Delta t=F(c)-F(d)$ for all
$c,d\in\mathbb{T}$. The delta integral has the following property:
\begin{equation*}
\int_t^{\sigma(t)}f(\tau)\Delta\tau=\mu(t)f(t).
\end{equation*}

A function $G: \mathbb{T} \rightarrow \mathbb{R} $ is called a
\emph{nabla antiderivative} of $g: \mathbb{T} \rightarrow \mathbb{R}$,
provided $G^{\nabla}(t)=g(t)$ holds for all $t \in
\mathbb{T}_{k}$. Then, the nabla integral of $g$ is defined by
$\int^b_a g(t)\nabla t=G(b)-G(a)$.

Let $\mathbb{T}$ be a time scale, and $t$, $s \in \mathbb{T}$.
Following \cite{RS}, we define $\mu_{t s} = \sigma(t)-s$, $\eta_{t
s} = \rho(t)-s$, and $f^{\diamond_{\alpha}}(t)$ to be the value, if
one exists, such that for all $\epsilon >0$ there is a neighborhood
$U$ of $t$ such that for all $s \in U$
\begin{equation*}
\left| \alpha \left[f^\sigma(t) - f(s)\right] \eta_{t s} +
(1-\alpha) \left[f^\rho(t) - f(s) \right] \mu_{t s} -
f^{\diamond_{\alpha}}(t) \mu_{t s} \eta_{t s} \right| < \epsilon
\left|\mu_{t s} \eta_{t s}\right| \, .
\end{equation*}
A function $f$ is said diamond-$\alpha$ differentiable on
$\mathbb{T}^k_k$ provided $f^{\diamond_{\alpha}}(t)$ exists for all
$t \in \mathbb{T}^k_k$. Let $0 \leq \alpha \leq 1$. If $f(t)$ is
differentiable on $t \in \mathbb{T}^k_k$ both in the delta and nabla
senses, then $f$ is diamond-$\alpha$ differentiable at $t$ and the
dynamic derivative $f^{\diamond_{\alpha}}(t)$ is given by
\begin{equation}
\label{eq:defSmp} f^{\diamond_{\alpha}}(t)= \alpha
f^{\Delta}(t)+(1-\alpha)f^{\nabla}(t)
\end{equation}
(see \cite[Theorem~3.2]{RS}). Equality \eqref{eq:defSmp} is given as
definition of $f^{\diamond_{\alpha}}(t)$ in \cite{sfhd}. The
diamond-$\alpha$ derivative reduces to the standard $\Delta$
derivative for $\alpha =1$, or the standard $\nabla$ derivative for
$\alpha =0$. On the other hand, it represents a ``weighted dynamic
derivative'' for $\alpha \in (0,1)$. Furthermore, the combined
dynamic derivative offers a centralized derivative formula on any
uniformly discrete time scale $\mathbb{T}$ when
$\alpha=\frac{1}{2}$.

Let $f, g: \mathbb{T} \rightarrow \mathbb{R}$ be diamond-$\alpha$
differentiable at $t \in \mathbb{T}^k_k$. Then (\textrm{cf.}
\cite[Theorem~2.3]{sfhd}),

\begin{itemize}

\item[(i)] $f+g: \mathbb{T} \rightarrow \mathbb{R}$ is
diamond-$\alpha$ differentiable at $t \in \mathbb{T}^k_k$ with
$$ (f+g)^{\diamond_{\alpha}}(t)=
f^{\diamond_{\alpha}}(t)+g^{\diamond_{\alpha}}(t);
$$

\item[(ii)] For any constant $c$, $cf: \mathbb{T} \rightarrow
\mathbb{R}$
 is diamond-$\alpha$ differentiable at $t \in \mathbb{T}^k_k$ with
$$
(cf)^{\diamond_{\alpha}}(t)= cf^{\diamond_{\alpha}}(t);
$$

\item[(iii)] $fg: \mathbb{T} \rightarrow \mathbb{R}$ is
diamond-$\alpha$ differentiable at $t \in \mathbb{T}^k_k$ with

$$
(fg)^{\diamond_{\alpha}}(t)= f^{\diamond_{\alpha}}(t)g(t)+ \alpha
f^{\sigma}(t)g^{\Delta}(t) +(1-\alpha) f^{\rho}(t)g^{\nabla}(t).
$$

\end{itemize}


\section{Main results}
\label{sec:MR}

Let $a, b \in \mathbb{T}$, and $h: \mathbb{T} \rightarrow
\mathbb{R}$. The diamond-$\alpha$ integral of $h$ from $a$
to $b$ is defined in \cite{RS,sfhd} by
\begin{equation}
\label{eq:DAI:JF} \int_{a}^{b}h(\tau) \diamond_{\alpha} \tau =
\alpha \int_{a}^{b}h(\tau) \Delta \tau +(1- \alpha)
\int_{a}^{b}h(\tau) \nabla \tau, \quad 0 \leq \alpha \leq 1,
\end{equation}
provided that there exist delta and nabla integrals of $h$ on
$\mathbb{T}$. In \S\ref{subsec:RI} we introduce a more basic notion of diamond-$\alpha$ integral. We use a Darboux approach without the need to define previously delta and nabla integrals.
Improper integrals are introduced in \S\ref{subsec:II}.
We end with \S\ref{subsec:MVT}, proving some generalizations of the mean value theorems on time scales via diamond-$\alpha$ derivatives and integrals. Moreover, a new notion of local extremum on time scales is proposed, which leads to a
diamond-alpha first order optimality condition more similar to the classical condition ($\mathbb{T}=\mathbb{R}$) than those of delta or nabla derivatives (\textrm{cf.} \cite{livro}).

\subsection{The Riemann diamond-$\alpha$ integral}
\label{subsec:RI}

Let $\mathbb{T}$ be a one-dimensional time scale, $a, b\in
 \mathbb{T}$, $a<b$ and $[a,b]$ a closed, bounded interval in
$\mathbb{T}$. A partition of $[a,b]$ is any finite ordered subset
\begin{equation*}
    P=\{t_{0},t_{1},\ldots,t_{n}\}\subset[a,b]\, ,
\end{equation*}
where $a=t_{0}<t_{1}<\cdots<t_{n}=b$. The number $n$ depends on the
particular partition, so we have $n=n(P)$. We denote by
$\mathcal{P}=\mathcal{P}(a,b)$ the set of all partitions of $[a,b]$.
Let $f$ be a real-valued bounded function on $[a,b]$. We set:
\begin{equation*}
 \overline{M}=\sup\{f(t):t\in[a,b)\}\, , \quad
 \overline{m}=\inf\{f(t):t\in[a,b)\}\, ,
 \end{equation*}
\begin{equation*}
 \underline{M}=\sup\{f(t):t\in(a,b]\}\, , \quad \underline{m}=\inf\{f(t):t\in(a,b]\}\, ,
\end{equation*}
and for $1\leq i \leq n$
\begin{equation*}
 \overline{M_{i}}=\sup\{f(t):t\in[t_{i-1},t_{i})\}\, , \quad
 \overline{m_{i}}=\inf\{f(t):t\in[t_{i-1},t_{i})\}\, ,
 \end{equation*}
\begin{equation*}
 \underline{M_{i}}=\sup\{f(t):t\in(t_{i-1},t_{i}]\}\, , \quad \underline{m_{i}}=\inf\{f(t):t\in(t_{i-1},t_{i}]\}\,
 .
\end{equation*}
Let $\alpha\in[0,1]\subset \R$. The upper Darboux
$\diamond_{\alpha}$-sum $U(f,P)$ and the lower Darboux
$\diamond_{\alpha}$-sum $L(f,P)$ of $f$ with respect to $P$ are
defined respectively by
\begin{equation*}
    U(f,P)=\sum_{i=1}^{n}(\alpha\overline{M_{i}}+(1-\alpha)\underline{M_{i}})(t_{i}-t_{i-1})\,
    ,
\end{equation*}
\begin{equation*}
    L(f,P)=\sum_{i=1}^{n}(\alpha\overline{m_{i}}+(1-\alpha)\underline{m_{i}})(t_{i}-t_{i-1})\,
    .
\end{equation*}
Note that
\begin{equation*}
    U(f,P)\leq
    \sum_{i=1}^{n}(\alpha\overline{M}+(1-\alpha)\underline{M})(t_{i}-t_{i-1})
    =(\alpha\overline{M}+(1-\alpha)\underline{M})(b-a)
\end{equation*}
and
\begin{equation*}
    L(f,P)\geq
    \sum_{i=1}^{n}(\alpha\overline{m}+(1-\alpha)\underline{m})(t_{i}-t_{i-1})
    =(\alpha\overline{m}+(1-\alpha)\underline{m})(b-a) \, .
\end{equation*}
Thus, we have:
\begin{equation}\label{sd}
   (\alpha\overline{m}+(1-\alpha)\underline{m})(b-a)\leq L(f,P)\leq U(f,P)\leq
    (\alpha\overline{M}+(1-\alpha)\underline{M})(b-a).
\end{equation}
The upper Darboux $\diamond_{\alpha}$-integral $U(f)$ of $f$ from
$a$ to $b$ is defined by
\begin{equation*}
    U(f)=\inf\{U(f,P): P\in \mathcal{P}(a,b)\}
\end{equation*}
and the lower Darboux $\diamond_{\alpha}$-integral $L(f)$ of $f$
from $a$ to $b$ is defined by
\begin{equation*}
    L(f)=\sup\{L(f,P): P\in \mathcal{P}(a,b)\}\, .
\end{equation*}
In view of \eqref{sd}, $U(f)$ and $L(f)$ are finite real numbers.

\begin{df}
We say that $f$ is $\diamond_{\alpha}$-integrable from $a$ to $b$
(or on $[a,b]$) provided $L(f)=U(f)$. In this case, we write
$\int_{a}^{b}f(t)\diamond_{\alpha}t$ for this common value. We call
this integral the Darboux $\diamond_{\alpha}$-integral.
\end{df}

Let $\overline{U}(f)$ and $\overline{L}(f)$ denote the upper and the lower Darboux $\Delta$-integral of $f$ from $a$ to $b$, respectively; $\underline{U}(f)$ and $\underline{L}(f)$ denote
the upper and the lower Darboux $\nabla$-integral of $f$ from $a$ to $b$, respectively. Given the
construction of $U(f)$ and $L(f)$, the equality
\eqref{eq:DAI:JF} follows from the properties of supremum and infimum.

\begin{Cor}
If $f$ is $\Delta$-integrable from $a$ to $b$ and
$\nabla$-integrable from $a$ to $b$, then it is $\diamond_{\alpha}$-integrable from $a$ to $b$ and
\begin{equation*}
\int_{a}^{b}f(t)\diamond_{\alpha}t=\alpha\int_{a}^{b}f(t)\Delta
t+(1-\alpha)\int_{a}^{b}f(t)\nabla t.
\end{equation*}
\end{Cor}

Now, suppose that $f$ is $\diamond_{\alpha}$-integrable from $a$ to
$b$. Then $L(f)=U(f)$ and
\begin{eqnarray*}
\alpha\overline{U}(f)+(1-\alpha)\underline{U}(f)=\alpha\overline{L}(f)+(1-\alpha)\underline{L}(f),\\
\alpha(\overline{U}(f)-\overline{L}(f))=(1-\alpha)(\underline{L}(f)-\underline{U}(f)).
\end{eqnarray*}
Since $\overline{U}(f)\geq \overline{L}(f)$ and
$\underline{U}(f)\geq \underline{L}(f)$, we get the following
result.
\begin{Cor}\label{essential}
Let $f$ be $\diamond_{\alpha}$-integrable from $a$ to $b$.

(i) If $\alpha =1$, then $f$ is $\Delta$-integrable from $a$ to $b$.

(ii) If $\alpha =0$, then $f$ is $\nabla$-integrable from $a$ to
$b$.

(iii) If $0< \alpha <1$, then $f$ is $\Delta$-integrable and
$\nabla$-integrable from $a$ to $b$.
\end{Cor}

\begin{Example}
Note that the strict inequalities in (iii) of the above corollary
are necessary. Consider the function $f(t)=\frac{1}{t}$ and the time scale $\mathbb{T}= \mathbb{Z}$.
We have $\int_{-1}^{0}f(t)\Delta
t=-1$ and $\int_{0}^{1}f(t)\nabla t=1$. However, both $\int_{0}^{1}f(t)\Delta t$ and
$\int_{-1}^{0}f(t)\nabla t$ do not exist.
\end{Example}

The following theorems may be showed in the same way as Theorem~5.5 and Theorem~5.6 in \cite{BG}.

\begin{Th}\label{someP}
If $U(f,P)=L(f,P)$ for some $P\in\mathcal{P}(a,b)$, then the
function $f$ is $\diamond_{\alpha}$-integrable from $a$ to $b$
and $\int_{a}^{b}f(t)\diamond_{\alpha}t=U(f,P)=L(f,P)$.
\end{Th}

\begin{Th}\label{epsilon}
A bounded function $f$ on $[a,b]$ is $\diamond_{\alpha}$-integrable
if and only if for each $\varepsilon >0$ there exists
$P\in\mathcal{P}(a,b)$ such that $U(f,P)-L(f,P)< \varepsilon$.
\end{Th}

\begin{Lemma}[\cite{BG}]
\label{delta}
For every $\delta >0$ there exists some partition $P\in
\mathcal{P}(a,b)$ given by $a=t_{0}< t_{1}<\cdots<t_{n}=b$ such that
for each $i\in\{1,2,\ldots,n\}$ either $t_{i}-t_{i-1}\leq \delta$ or
$t_{i}-t_{i-1}> \delta$ and $\rho(t_{i})=t_{i-1}$.
\end{Lemma}

We denote by $\mathcal{P}_{\delta}=\mathcal{P}_{\delta}(a,b)$ the
set of all $P\in \mathcal{P}(a,b)$ that possess the property
indicated in Lemma \ref{delta}.

\begin{Th}
A bounded function $f$ on $[a,b]$ is $\diamond_{\alpha}$-integrable
if and only if for each $\varepsilon >0$ there exists $\delta>0$
such that
\begin{equation}
\label{eq:imp:i}
P\in\mathcal{P}_{\delta}(a,b) \Rightarrow
U(f,P)-L(f,P)< \varepsilon \, .
\end{equation}
\end{Th}
\begin{proof}
By Theorem \ref{epsilon} condition \eqref{eq:imp:i} implies integrability.
Conversely, suppose that $f$ is $\diamond_{\alpha}$-integrable from
$a$ to $b$. If $\alpha=1$ or $\alpha=0$, then f is
$\Delta$-integrable from $a$ to $b$ or $\nabla$-integrable from $a$
to $b$. Therefore condition \eqref{eq:imp:i} holds (see \cite{BG}). Now, let
$0<\alpha <1$. By Corollary \ref{essential}, $f$ is
$\Delta$-integrable and $\nabla$-integrable from $a$ to $b$. According to \cite[Theorem~5.9]{BG}, for each $\varepsilon >0$
there exist $\delta_{1}>0$ and $\delta_{2}>0$ such that
$P_{1}\in\mathcal{P}_{\delta_{1}}(a,b)$ implies
$\overline{U}(f,P_{1})-\overline{L}(f,P_{1})< \varepsilon$ and
$P_{2}\in\mathcal{P}_{\delta_{2}}(a,b)$ implies
$\underline{U}(f,P_{2})-\underline{L}(f,P_{2})< \varepsilon$.
If $P\in\mathcal{P}_{\delta}(a,b)$ where $\delta=min\{\delta_{1},
\delta_{2}\}$, then we have $U(f,P)-L(f,P)=\alpha
\overline{U}(f,P)+(1-\alpha)\underline{U}(f,P)-\alpha
\overline{L}(f,P)-(1-\alpha)\underline{L}(f,P)< \varepsilon$.
\end{proof}

Let $f$ be a bounded function on $[a,b]$ and let
$P\in\mathcal{P}(a,b)$ be given by $a=t_{0}<t_{1}<\cdots<t_{n}=b$. For
$1\leq i \leq n$, choose arbitrary points $\overline{\xi} _{i}$ in
$[t_{i-1},t_{i})$, $\underline{\xi}_{i}$ in $(t_{i-1},t_{i}]$, and
form the sum
$$
S= \Sigma_{i=1}^{n}(\alpha
f(\overline{\xi}_{i})+(1-\alpha)f(\underline{\xi}_{i}))(t_{i}-t_{i-1}) \, .
$$
We call $S$ a Riemann $\diamond_{\alpha}$-sum of $f$ corresponding to $P\in\mathcal{P}(a,b)$. We say that $f$ is Riemann
$\diamond_{\alpha}$-integrable from $a$ to $b$ if there exists a
real number $I$ with the following property: for each $\varepsilon
>0$ there exists $\delta>0$ such that
$P\in\mathcal{P}_{\delta}(a,b)$ implies $|S-I|< \varepsilon$
independent of the choice of $\overline{\xi} _{i}$,
$\underline{\xi}_{i}$ for $1\leq i \leq n$. The number $I$ is called
the Riemann $\diamond_{\alpha}$-integral of $f$ from $a$ to $b$.

\begin{Th}\label{Riemann}
If $f$ is Riemann $\Delta$-integrable and Riemann $\nabla$-integrable
from $a$ to $b$, then it is Riemann $\diamond_{\alpha}$-integrable
from $a$ to $b$ and $I=\alpha\int_{a}^{b}f(t)\Delta
t+(1-\alpha)\int_{a}^{b}f(t)\nabla t$.
\end{Th}
\begin{proof}
Assume that $f$ is Riemann $\Delta$-integrable and Riemann
$\nabla$-integrable from $a$ to $b$. Then, for each $\varepsilon >0$
there exist $\delta_{1}>0$  and $\delta_{2}>0$ such that
$P_{1}\in\mathcal{P}_{\delta_{1}}(a,b)$ implies
$|\overline{S}-\int_{a}^{b}f(t)\Delta t|< \varepsilon$, and
$P_{2}\in\mathcal{P}_{\delta_{2}}(a,b)$ implies
$|\underline{S}-\int_{a}^{b}f(t)\nabla t|< \varepsilon$, where
$\overline{S}$ is the Riemann $\Delta$-sum of $f$ corresponding to
$P_{1}$, and $\underline{S}$ is the Riemann $\nabla$-sum of $f$
corresponding to $P_{2}$. Now, if $P\in\mathcal{P}_{\delta}(a,b)$
with $\delta=min\{\delta_{1}, \delta_{2}\}$, then we have
\begin{equation*}
\begin{split}
\Bigl|S-\alpha\int_{a}^{b}f(t)\Delta t &- (1-\alpha)\int_{a}^{b}f(t)\nabla t\Bigr| \\
&=\left|\alpha
\overline{S}+(1-\alpha)\underline{S}-\alpha\int_{a}^{b}f(t)\Delta
t-(1-\alpha)\int_{a}^{b}f(t)\nabla t\right|\\
&\leq
\left|\alpha\overline{S}-\alpha\int_{a}^{b}f(t)\Delta
t\right|+\left|(1-\alpha)\underline{S}-(1-\alpha)\int_{a}^{b}f(t)\nabla t\right|\\
&\leq \varepsilon \, .
\end{split}
\end{equation*}
Thus, $f$ is Riemann $\diamond_{\alpha}$-integral
from $a$ to $b$ and
$I=\alpha\int_{a}^{b}f(t)\Delta
t+(1-\alpha)\int_{a}^{b}f(t)\nabla t$.
\end{proof}

By construction of the $\diamond_{\alpha}$-Riemann sum $S$,
the following theorem may be proved in much the same way as
\cite[Theorem~5.11]{BG}.

\begin{Th}\label{integralRD}
A bounded function $f$ on $[a,b]$ is Riemann
$\diamond_{\alpha}$-integrable if and only if it is Darboux
$\diamond_{\alpha}$-integrable, in which case the values of the integrals are equal.
\end{Th}

We define
\begin{equation*}
\int_{a}^{a}f(t)\diamond_{\alpha}t=0
\end{equation*}
and
\begin{equation*}
    \int_{a}^{b}f(t)\diamond_{\alpha}t=-\int_{b}^{a}f(t)\diamond_{\alpha}t,\,
    \ a>b.
\end{equation*}

\begin{Cor}
Let $f$ be Riemann $\diamond_{\alpha}$-integrable from $a$ to $b$.

(i) If $\alpha =1$, then $f$ is Riemann $\Delta$-integrable from $a$
to $b$.

(ii) If $\alpha =0$, then $f$ is Riemann $\nabla$-integrable from
$a$ to $b$.

(iii) If $0< \alpha <1$, then $f$ is is Riemann $\Delta$-integrable
and Riemann $\nabla$-integrable from $a$ to $b$.
\end{Cor}

\begin{Th}
Let $f: \mathbb{T}\rightarrow \R$ and let $t\in \mathbb{T}$. Then,

(i) $f$ is integrable from $t$ to $\sigma(t)$ and
\begin{equation}\label{i}
\int_{t}^{\sigma(t)}f(s)\diamond_{\alpha}s=\mu(t)(\alpha
f(t)+(1-\alpha)f^{\sigma} (t)) \, ;
\end{equation}
(ii) $f$ is integrable from $\rho(t)$ to $t$ and
\begin{equation*}
\int_{\rho(t)}^{t}f(s)\diamond_{\alpha}s=\eta(t)(\alpha
f^{\rho}(t)+(1-\alpha)f(t)).
\end{equation*}
\end{Th}
\begin{proof}
(i) If $\sigma(t)=t$, then $\mu(t)=0$ and equality \eqref{i} is
obvious. If $\sigma(t)>t$, then $\mathcal{P}(t,\sigma(t))$ contains
only one element given by $t=s_{0}<s_{1}=\sigma(t)$. Since
$[s_{0},s_{1})={t}$ and $(s_{0},s_{1}]={\sigma(t)}$, we have
$U(f,P)=\alpha
f(t)(\sigma(t)-t)+(1-\alpha)f^{\sigma}(t)(\sigma(t)-t)=L(f,P)$. By Theorem~\ref{someP}, $f$ is $\diamond_{\alpha}$-integrable from $t$
to $\sigma(t)$ and \eqref{i} holds. Proof of (ii)
is done in a similar way.
\end{proof}

\begin{Cor}
Let $a, b\in \mathbb{T}$ and $a<b$. Then we have the following:

(i) If $\mathbb{T}= \R$, then a bounded function $f$ on $[a,b]$ is
$\diamond_{\alpha}$-integrable from $a$ to $b$  if and only if is
Riemann integrable on $[a,b]$ in the classical sense, and in this
case $\int_{a}^{b}f(t)\diamond_{\alpha}t=\int_{a}^{b}f(t)dt$.

(ii) If $\mathbb{T}= \mathbb{Z}$, then every function $f$ defined on
$\mathbb{Z}$ is $\diamond_{\alpha}$-integrable from $a$ to $b$, and
$\int_{a}^{b}f(t)\diamond_{\alpha}t=\sum_{t=a+1}^{b-1}f(t)+\alpha
f(a)+(1-\alpha)f(b)$.

(iii) If $\mathbb{T}= h\mathbb{Z}$, then then every function $f$
defined on $h\mathbb{Z}$ is $\diamond_{\alpha}$-integrable from $a$
to $b$, and
$\int_{a}^{b}f(t)\diamond_{\alpha}t=\sum_{k=\frac{a}{h}+1}^{\frac{b}{h}-1}f(kh)h+\alpha
f(a)h+(1-\alpha)f(b)h$.
\end{Cor}

The following results are straightforward consequences of Theorems~\ref{Riemann} and \ref{integralRD} and properties of the
Riemann delta (nabla) integral:

\begin{itemize}
\item[1.] Let $a, b\in \mathbb{T}$ and $a<b$. Every constant function
$f:\mathbb{T}\rightarrow \R$ is $\diamond_{\alpha}$-integrable from
$a$ to $b$ and $\int_{a}^{b}f(t)\diamond_{\alpha}t=c(b-a)$.

\item[2.] Every monotone function $f:\mathbb{T}\rightarrow \R$ on $[a,b]$ is
$\diamond_{\alpha}$-integrable from $a$ to $b$.

\item[3.] Every continuous function $f:\mathbb{T}\rightarrow \R$ on $[a,b]$ is
$\diamond_{\alpha}$-integrable from $a$ to $b$.

\item[4.] Every bounded function $f:\mathbb{T}\rightarrow \R$ on $[a,b]$ with
only finitely many discontinuity points is
$\diamond_{\alpha}$-integrable from $a$ to $b$.

\item[5.] Every regulated  function $f:\mathbb{T}\rightarrow \R$ on $[a,b]$ is
$\diamond_{\alpha}$-integrable from $a$ to $b$.

\item[6.] Let $f:\mathbb{T}\rightarrow \R$
be a bounded function on $[a,b]$
that is $\diamond_{\alpha}$-integrable from $a$ to $b$. Then, $f$ is
$\diamond_{\alpha}$-integrable on every subinterval $[c,d]$ of $[a,b]$.

\item[7.]Let $f,g$ be $\diamond_{\alpha}$-integrable from $a$ to $b$ and
$c\in \R$, $d\in\mathbb{T}$, $a<d<b$. Then:\\
(i) $cf$ is $\diamond_{\alpha}$-integrable from $a$ to $b$ and
$\int_{a}^{b}(cf)(t)\diamond_{\alpha}t=c\int_{a}^{b}f(t)\diamond_{\alpha}t$,\\
(ii) $f+g$ is $\diamond_{\alpha}$-integrable from $a$ to $b$ and
$\int_{a}^{b}(f+g)(t)\diamond_{\alpha}t
=\int_{a}^{b}f(t){\diamond_{\alpha}t}+\int_{a}^{b}g(t)\diamond_{\alpha}t$,\\
(iii) $fg$ is $\diamond_{\alpha}$-integrable from $a$ to $b$,\\
(iv)
$\int_{a}^{b}f(t)\diamond_{\alpha}t=\int_{a}^{d}f(t)\diamond_{\alpha}t+\int_{d}^{b}f(t)\diamond_{\alpha}t$.

\item[8.] If $f,g$ are $\diamond_{\alpha}$-integrable from $a$ to $b$ and if
$f(t)\leq g(t)$ for all $t\in [a,b]$, then
$\int_{a}^{b}f(t)\diamond_{\alpha}t\leq\int_{a}^{b}f(t)\diamond_{\alpha}t$.

\item[9.] If $f$ is $\diamond_{\alpha}$-integrable from $a$ to $b$, then so
is $|f|$. Moreover, $|\int_{a}^{b}f(t)\diamond_{\alpha}t|\leq
\int_{a}^{b}|f(t)|\diamond_{\alpha}t$.
\end{itemize}


\subsection{Improper integrals}
\label{subsec:II}

Let $\mathbb{T}$ be a time scale, $a \in
 \mathbb{T}$. Throughout this section we assume that there exists a
subset
\begin{equation*}
\{t_{k}:k\in \N_{0}\}\subset \mathbb{T}\, , \quad
a=t<t_{1}<t_{2}<\cdots\, , \quad \lim_{k\rightarrow\infty}t_{k}=\infty.
\end{equation*}
Let us suppose that the real-valued function $f$ is defined on $[a,\infty)$
and is Riemann $\diamond_{\alpha}$-integrable from $a$ to any point
$A\in \mathbb{T}$ with $A\geq a$. If the integral
\begin{equation*}
    F(A)=\int_{a}^{A}f(t)\diamond_{\alpha}t
\end{equation*}
approaches a finite limit as $A\rightarrow \infty$, we call that
limit the improper diamond-${\alpha}$ integral of first kind of $f$
from $a$ to $\infty$ and write
\begin{equation}\label{impint}
\int_{a}^{\infty}f(t)\diamond_{\alpha}t
=\lim_{A\rightarrow\infty}\left(\int_{a}^{A}f(t)\diamond_{\alpha}t\right) .
\end{equation}
In such case we say that the improper integral
$\int_{a}^{\infty}f(t)\diamond_{\alpha}t$ exists or that it is
convergent. If the limit \eqref{impint} does not exist, we say that
the improper integral does not exist or is divergent. Note that
\begin{equation*}
\int_{a}^{\infty}f(t)\diamond_{\alpha}t=\lim_{A\rightarrow\infty}
\left(\alpha
\int_{a}^{A}f(t)\triangle t+(1-\alpha)\int_{a}^{A}f(t)\nabla t\right) \, .
\end{equation*}

\begin{Cor}
If improper integrals $\int_{a}^{\infty}f(t)\triangle t$ and
$\int_{a}^{\infty}f(t)\nabla t$ exist, then
\begin{equation*}
\int_{a}^{\infty}f(t)\diamond_{\alpha}t=\alpha
\int_{a}^{\infty}f(t)\triangle
t+(1-\alpha)\int_{a}^{\infty}f(t)\nabla t.
\end{equation*}
\end{Cor}
On the other hand, the improper
diamond-${\alpha}$ integral may exist even if improper delta and
nabla integrals do not exist.
\begin{Example}
Consider the function
\begin{gather*}
f(t)=
\begin{cases}
1 & \text{ if } t=2l \\
-1 & \text{ if } t=2l+1,
\end{cases}
\quad l\in \mathbb{T} ,
\end{gather*}
on the time scale $\mathbb{T}= \N\cup\{0\}$. We have
\begin{equation*}
\int_{0}^{\infty}f(t)\triangle t=\sum_{k=0}^{\infty}f(k)\, , \quad
\int_{0}^{\infty}f(t)\nabla t=\sum_{k=0}^{\infty}f(k+1).
\end{equation*}
Hence, the improper integrals $\int_{0}^{\infty}f(t)\triangle t$ and
$\int_{0}^{\infty}f(t)\nabla t$ do not exist. On the other hand, for
$\alpha=1/2$ we have
\begin{equation*}
\int_{a}^{A}f(t)\diamond_{\alpha}t=\frac{1}{2}\left(\sum_{k=0}^{A-1}f(k)+
\sum_{k=0}^{A-1}f(k+1)\right)=\frac{1}{2}\sum_{k=0}^{A-1}(f(k)+f(k+1))=0
\end{equation*}
and
\begin{equation*}
\int_{a}^{\infty}f(t)\diamond_{\alpha}t
=\lim_{A\rightarrow\infty}\left(\int_{a}^{A}f(t)\diamond_{\alpha}t\right)=0.
\end{equation*}
\end{Example}


\subsection{Diamond-alpha Fermat's and Mean Value Theorems}
\label{subsec:MVT}

Theorem~\ref{MVTin1}
and Theorem~\ref{MVTin2} are exact analogies of mean value theorems for delta (nabla) integrals and there are no differences in proofs of these theorems. However,
the formulation of the mean value theorems~\ref{thm:mvt1} and \ref{thm:mvt2} below, for the diamond-${\alpha}$
derivative, provide a generalization more similar to the classical results than the ones previously proved for the delta or nabla derivatives (\textrm{cf.} \cite{BGd}).

\begin{Th}\label{MVTin1}
Let $f$ and $g$ be bounded and $\diamond_{\alpha}$-integrable
functions from $a$ to $b$, and let $g$ be nonnegative (or
nonpositive) on $[a,b]$. Let $m$ and $M$ be the infimum and
supremum, respectively, of the function $f$ on $[a,b]$. Then, there
exists a real number $K$ satisfying the inequalities $m\leq K \leq
M$ such that
\begin{equation*}
\int_{a}^{b}f(t)g(t)\diamond_{\alpha}t=K
\int_{a}^{b}g(t)\diamond_{\alpha}t.
\end{equation*}
\end{Th}

\begin{Th}\label{MVTin2}
Let $f$  be bounded and $\diamond_{\alpha}$-integrable function on
$[a,b]$. Let $m$ and $M$ be the infimum and supremum, respectively,
of
the function  $F(s)=\int_{a}^{t}f(s)\diamond_{\alpha}s$ on $[a,b]$. We have: \\
(i) If a function $g$ is non-increasing with $g(t)\geq 0$ on
$[a,b]$, then there is a number $K$ such that $m\leq K \leq M$ and
\begin{equation*}
\int_{a}^{b}f(t)g(t)\diamond_{\alpha}t=Kg(a).
\end{equation*}
(ii) If $g$ is any monotonic function on $[a,b]$, then there is a
number $K$ such that $m\leq K \leq M$ and
\begin{equation*}
\int_{a}^{b}f(t)g(t)\diamond_{\alpha}t=[g(a)-g(b)]K+g(b)\int_{a}^{b}f(t)\diamond_{\alpha}t.
\end{equation*}
\end{Th}

We now define a local maximum of a real function defined on a time
scale $\mathbb{T}$. The definition of a local minimum is done in a
similar way.

\begin{df}
We say that a function $f:\mathbb{T}\rightarrow \R$ assumes its
local maximum at $t_{0}\in\mathbb{T}_{k}^{k}$ provided \\
(i) if $t_{0}$ is scattered,  then $f(\sigma (t_{0}))\leq f(t_{0})$
and $f(\rho(t_{0}))\leq f(t_{0})$; \\
(ii) if $t_{0}$ is dense, then there is a neighborhood $U$ of
$t_{0}$ such that $f(t)\leq f(t_{0})$ for all $t\in U$;\\
(iii) if $t_{0}$ is left-scattered and right-dense, then
$f(\rho(t_{0}))\leq f(t_{0})$ and there is a neighborhood $U$ of
$t_{0}$ such that $f(t)\leq f(t_{0})$ for all $t\in U$ with $t>t_{0}$;\\
(iv) if $t_{0}$ is right-scattered and left-dense, then $f(\sigma
(t_{0}))\leq f(t_{0})$
 and there is a neighborhood $U$ of
$t_{0}$ such that $f(t)\leq f(t_{0})$ for all $t\in U$ with
$t<t_{0}$.
\end{df}

Theorem~\ref{thm:ft} permits to introduce the notion
of critical point in a similar way as done in classical calculus. We remark that equality $f^{\diamond_{\alpha}}(t_{0})=0$ in Theorem~\ref{thm:ft} does not always hold for the delta ($\alpha = 1$) or nabla ($\alpha = 0$) cases (\textrm{cf.} \cite{BGd}).

\begin{Th}[diamond-alpha Fermat's theorem for stationary points]
\label{thm:ft}
Suppose f assumes its local extremum at $t_{0}\in\mathbb{T}_{k}^{k}$
and $f$ is delta and nabla differentiable at $t_{0}$. Then, there exists $\alpha\in [0,1]$ such that $f^{\diamond_{\alpha}}(t_{0})=0$.
\end{Th}

\begin{proof}
Suppose that f assumes its local maximum at
$t_{0}\in\mathbb{T}_{k}^{k}$. Then, we have $f^{\triangle}(t_{0})\leq
0$ and $f^{\nabla}(t_{0})\geq 0$. If $f^{\triangle}(t_{0})=0$ ($f^{\nabla}(t_{0})=0$), then we put $\alpha=1$ ($\alpha=0$).
Therefore, we can assume that $f^{\triangle}(t_{0})< 0$ and
$f^{\nabla}(t_{0}) >0$. Setting
\begin{equation*}
\alpha=\frac{f^{\nabla}(t_{0})}{f^{\nabla}(t_{0})-f^{\triangle}(t_{0})},
\end{equation*}
it is easy to see that $0<\alpha <1$, and we obtain the intended result.
\end{proof}

\begin{Example}
Let $\mathbb{T} = \{-1, 0,1,2,3,4\}$, and $f$ be defined by
$f(-1)=f(0) = 5$, $f(1) = 0$, $f(2) = 1$, and $f(3) = f(4) = 3$. The
point 1 is a local minimizer, points 0 and 3 are local maximizers,
and point 2 is neither a minimizer neither a maximizer (as well as
-1 and 4, by definition). The delta-derivative is only zero at point
3 while the nabla-derivative is zero only at zero. According with
Theorem~\ref{thm:ft}, at all extremizers there exists an alpha such
that the diamond-alpha derivative vanishes. Indeed,
$f^{\diamond_{0}}(0)=0$, $f^{\diamond_{5/6}}(1)=0$, and
$f^{\diamond_{1}}(3)=0$. Observe that $f^{\diamond_{\alpha}}(2) \ne
0$ for all $\alpha\in [0,1]$.
\end{Example}

\begin{Th}[diamond-alpha Rolle's mean value theorem]
\label{thm:mvt1}
Let $f$ be a continuous function on $[a,b]$ that is delta and nabla
differentiable on $(a,b)$ with $f(a)=f(b)$. Then, there exists
$\alpha\in [0,1]$ and $c\in (a,b)$ such that
$f^{\diamond_{\alpha}}(c)=0$.
\end{Th}
\begin{proof}
If $f=const$, then $f^{\diamond_{\alpha}}(c)=0$ for all $\alpha\in
[0,1]$ and $c\in (a,b)$. Hence, assume that $f$ is not the constant
function and $f(t)\geq f(a)$ for all $t\in [a,b]$. Since function
$f$ is continuous on the compact set $[a,b]$, $f$ assumes its
maximum $M>f(a)$. Therefore, there exists $c\in [a,b]$ such that
$M=f(c)$. As $f(a)=f(b)$, $a<c<b$, clearly $f$ assumes its local
maximum at $c$ and there exists $\alpha\in [0,1]$ such that
$f^{\diamond_{\alpha}}(c)=0$.
\end{proof}

The following mean value theorem is a generalization of Theorem~\ref{thm:mvt1}.

\begin{Th}[diamond-alpha Lagrange's mean value theorem]
\label{thm:mvt2} Let $f$ be a continuous function on $[a,b]$ that is
delta and nabla differentiable on $(a,b)$. Then, there exists
$\alpha\in [0,1]$ and $c\in (a,b)$ such that
\begin{equation*}
f^{\diamond_{\alpha}}(c)=\frac{f(b)-f(a)}{b-a} \, .
\end{equation*}
\end{Th}
\begin{proof}
Consider the function $g$
defined on $[a,b]$ by
\begin{equation*}
    g(t)=f(a)-f(t)+(t-a)\frac{f(b)-f(a)}{b-a}.
\end{equation*}
Clearly $g$ is continuous on $[a,b]$ and $\triangle$ and $\nabla$
differentiable on $(a,b)$. Also $g(a)=g(b)=0$. Hence, there exists
$\alpha\in [0,1]$ and $c\in (a,b)$ such that
$g^{\diamond_{\alpha}}(c)=0$. Since
\begin{multline*}
g^{\diamond_{\alpha}}(t)= \alpha
g^{\triangle}(t)+(1-\alpha)g^{\nabla}(t)\\
= \alpha\left(-f^{\triangle}
(t)+\frac{f(b)-f(a)}{b-a}\right)
+(1-\alpha)\left(-f^{\nabla}
(t)+\frac{f(b)-f(a)}{b-a}\right),
\end{multline*}
we conclude that
\begin{equation*}
0=-\alpha
f^{\triangle}(c)-(1-\alpha)f^{\nabla}(c)+\frac{f(b)-f(a)}{b-a}
\Leftrightarrow
f^{\diamond_{\alpha}}(c)=\frac{f(b)-f(a)}{b-a} \, .
\end{equation*}
\end{proof}

We end by proving a diamond-alpha Cauchy mean value theorem,
which is the more general form of the diamond-alpha
mean value theorem.

\begin{Th}[diamond-alpha Cauchy's mean value theorem]
\label{CTh}
Let $f$ and $g$ be continuous functions on $[a,b]$ that
are delta and nabla differentiable on $(a,b)$. Suppose that
$g^{\diamond_{\alpha}}(t)\neq 0$ for all $t\in (a,b)$ and all
$\alpha\in [0,1]$. Then, there exists $\bar{\alpha}\in [0,1]$ and
$c\in (a,b)$ such that
\begin{equation*}
\frac{f(b)-f(a)}{g(b)-g(a)}=\frac{f^{\diamond_{\bar{\alpha}}}(c)}{g^{\diamond_{\bar{\alpha}}}(c)}.
\end{equation*}
\end{Th}
\begin{proof}
Let us first observe that from the condition
$g^{\diamond_{\alpha}}(t)\neq 0$ for all $t\in (a,b)$ and all
$\alpha\in [0,1]$ it follows from Theorem~\ref{thm:mvt1}
that $g(b)\neq g(a)$. Hence, we may consider the function $F$ defined on $[a,b]$ by
\begin{equation*}
    F(t)=f(t)-f(a)-\frac{f(b)-f(a)}{g(b)-g(a)}[g(t)-g(a)].
\end{equation*}
Clearly, $F$ is continuous on $[a,b]$ and delta and nabla
differentiable on $(a,b)$. Also, $F(a)=F(b)$. Applying
Theorem~\ref{thm:mvt1} to the function $F$ and taking into account that
\begin{equation*}
\begin{split}
F^{\diamond_{\alpha}}(t) &= \alpha
F^{\triangle}(t)+(1-\alpha)F^{\nabla}(t)\\
&=\alpha
\left(f^{\triangle}(t)-\frac{f(b)-f(a)}{g(b)-g(a)}g^{\triangle}(t)\right)
+(1-\alpha)\left(f^{\nabla}(t)-\frac{f(b)-f(a)}{g(b)-g(a)}g^{\nabla}(t)\right) \, ,
\end{split}
\end{equation*}
we conclude that there exists $\bar{\alpha}\in [0,1]$
and $c\in (a,b)$ such that
\begin{equation*}
0=f^{\diamond_{\bar{\alpha}}}(c)-\frac{f(b)-f(a)}{g(b)-g(a)}g^{\diamond_{\bar{\alpha}}}(c).
\end{equation*}
Hence,
\begin{equation*}
f^{\diamond_{\bar{\alpha}}}(c)=\frac{f(b)-f(a)}{g(b)-g(a)}g^{\diamond_{\bar{\alpha}}}(c) \, ,
\end{equation*}
and dividing by $g^{\diamond_{\bar{\alpha}}}(c)$ we complete the
proof.
\end{proof}


\section*{Acknowledgments}

Research partially supported by the
{\it Centre for Research on Optimization and Control} (CEOC) from
the {\it Portuguese Foundation for Science and Technology} (FCT),
cofinanced by the European Community Fund FEDER/POCI 2010. The first author was also supported by KBN
under Bia{\l}ystok Technical University Grant S/WI/1/08.



\end{document}